\documentclass[11pt]{amsart}
\usepackage{amssymb}
\usepackage{xypic}
\xyoption{all}

\newtheorem{thm}{Theorem}[section] 

\newtheorem{cor}[thm]{Corollary}

\newtheorem{exam}[thm]{Example}

\newtheorem{lem}[thm]{Lemma}
\newtheorem{prop}[thm]{Proposition}

\newtheorem{ques}[thm]{Question}

\def\[{\left[}
\def\]{\right]}

\long\def\forget#1\forgotten{{}}

\def\N{{\mathbb{N}}}

\begin{document}

\title{Free Subalgbras of Graded Algebras}
\author{Jason P. Bell}
\thanks{The research of the first-named author  was supported by NSERC Grant 326532-2011.}
\address{ Department of Mathematics, University of Waterloo, 
Waterloo, ON, N2L  3G1, CANADA}
\email{jpbell@uwaterloo.ca}

\author{Be'eri Greenfeld}

\address{Department of Mathematics, Bar Ilan University, Ramat Gan 5290002, Israel}
\email{beeri.greenfeld@gmail.com}


\begin{abstract} Let $k$ be a field and let $A=\bigoplus_{n\ge 1}A_n$ be a positively graded $k$-algebra.  We recall that $A$ is graded nilpotent if for every $d\ge 1$, the subalgebra of $A$ generated by elements of degree $d$ is nilpotent.  We give a method of producing grading nilpotent algebras and use this to prove that over any base field $k$ there exists a finitely generated graded nilpotent algebra that contains a free $k$-subalgebra on two generators. 
\end{abstract}
\maketitle
\section{Introduction}

In recent years there has been a flurry of activity in ring theory around problems related to the Kurosh and K\"othe conjectures (see, for example, \cite{Kurosh_General, AgataTom, AgataEdmund, AgataEdmund2, Koethe_Survey, Primitive_Ideals_GrNil}). The Kurosh conjecture, which asserts that a finitely generated algebra that is algebraic over the base field is necessarily finite-dimensional, was disproved in 1964 by Golod and Shafarevich \cite{GolodShafarevich}, who constructed a counter-example to both this conjecture and to the Burnside problem, which is the group theoretic analogue of the Kurosh problem.  The K\"othe conjecture, on the other hand is still open.  It asks whether the sum of two nil left ideals in a ring is again a nil left ideal.  This was proved for uncountable base fields by Amitsur \cite{Amitsur_Koethe} in 1956, but the conjecture remains open in general. 

K\"othe's conjecture is equivalent to the statement that if $R$ is a nil ring then $R[x]$ is equal to its own Jacobson radical (see Krempa \cite{Krempa} for this and other equivalent statements). Towards better understanding this conjecture, Amitsur \cite{Amitsur_Question} conjectured that if $R$ is nil then in fact $R[x]$ is nil, too, which is a stronger conjecture. Again, Amitsur proved that this is true when $R$ is an algebra over an uncountable base field \cite{Amitsur_Polynomial}, but Smoktunowicz \cite{AmitsurAgata} later constructed counter-examples to Amitsur's conjecture over countable fields.  Her constructions have since provided impetus for a lot of the recent work around nil rings and questions inspired by K\"othe's conjecture and related problems.

We observe that if one takes a nil ring $R$ and one grades $R[x]$ by letting the elements of $R$ be of degree zero and letting the central variable $x$ have degree one, then while Smoktunowicz has shown that $R[x]$ need not be nil, we do have that $R[x]$ is \emph{graded nil}; that is, every homogeneous element of $R[x]$ is nilpotent.  Smoktunowicz \cite{AgataGrNil} and Regev \cite{Regev} attribute to Small and Zelmanov (asked in 2006) the question as to whether over uncountable base fields one has that finitely generated graded nil algebras must be nil. Smoktunowicz \cite{AgataGrNil} gave an example, showing that this question has a negative answer.  On the other hand, Regev \cite{Regev} showed that for finitely generated algebraic algebras over uncountable base fields, one has that the associated graded algebra (with respect to a standard filtration) is necessarily nil. We note that the opposite direction of the question is trivially true, but one has even more: the Jacobson radical of a graded ring is generated by homogeneous elements and is graded nil \cite{Bergman}.

Let $k$ be a field and let $A=A_1\oplus A_2 \oplus \cdots $ be a positively graded $k$-algebra; that is, $A$ is a $k$-algebra and each $A_i$ is a $k$-vector subspace of $A$ with the property that $A_i A_j \subseteq A_{i+j}$ for all $i,j\ge 1$. Following \cite{GLSZ}, we say that $A$ is \textit{graded nilpotent} if the algebra generated by any set of homogeneous elements of the same degree is nilpotent. Answering a question posed in \cite{GLSZ}, the first-named author and Madill \cite{Iterative} constructed finitely generated, graded nilpotent algebras which are not nilpotent. 

On the other hand, over a given countable base field, Smoktunowicz \cite{ML-Agata} constructed a nil algebra $R$ such that the polynomial ring in six variables, $x_1,\ldots ,x_6$, over $R$ contains a free algebra on two generators; in fact, one can take the generators to be homogeneous linear polynomials.  By again giving the elements of $R$ degree zero and giving $x_i$ degree $e_i\in \mathbb{N}^6$ (where $e_i$ is the vector with a $1$ in the $i$-th coordinate and zeros everywhere else), then we see that $R[x_1,\dots,x_6]$ is a graded nil algebra (with an $\mathbb{N}^6$-grading) that contains a free subalgebra on two generators.

At the moment, we do not know if there exists a nil algebra $R$ such that $R[x]$ contains a free subalgebra, since it is possible that as one adjoins variables in Smoktunowicz's construction that at some step one obtains an algebra that is neither nil nor contains a free algebra. 

Our main result is the following, which can be thought of as a continuation of Smoktunowicz's result \cite{ML-Agata} to the uncountable base field case.  (See Theorem \ref{thm: construction} for a more precise statement.)

\begin{thm} \label{Graded_Nil_Free_Subalgebra}
Let $k$ be a field. There exists a finitely generated positively graded $k$-algebra $A$, generated in degrees one and two, that is graded nilpotent and which contains a copy of the free algebra on two generators.  Moreover, the algebra $A$ can be taken to be a monomial algebra.
\end{thm}
We recall that an algebra is a monomial algebra if it is a quotient of a free algebra by an ideal generated by monomials in the generators. Monomial algebras are often useful in constructing counter-examples, because the family of finitely generated monomial algebras is very large (having the same size as the power set of $\mathbb{N}$) and because one can often employ combinatorial techniques to obtain ring theoretic information. Note that the constructed algebras $A$ are not Jacobson radical, since the Jacobson radical of a finitely generated monomial algebra is locally nilpotent \cite{JacLocNilp}.

Anick \cite{Anick} proved that a finitely presented monomial algebra of infinite Gelfand-Kirillov dimension must contain a free subalgebra on two generators that is generated by monomials. In addition, he asked whether a connected graded, finitely presented algebra (resp. an algebra with finite global dimension) of exponential growth must contain a homogeneous free subalgebra \cite[Question 1]{Anick}. In a similar vein, Smoktunowicz \cite{subalg} proved that finitely presented graded algebras with sufficiently sparse relations must contain free subalgebras generated by monomials.

We mention that in the context of graded (or monomial) algebras, one can distinguish between certain types of free subalgebras:
\begin{enumerate}
\item[($a$)] free subalgebras with no assumptions on the structure of the generators;
\item[($b$)] free subalgebras generated by homogeneous elements;
\item[($c$)] free subalgebras generated by monomials.
\end{enumerate}
We observe that $$(c)\implies (b)\implies (a).$$
Theorem \ref{Graded_Nil_Free_Subalgebra} shows in particular that for a graded algebra we do not have $(a)\implies (b)$. As an additional observation, we give an example that shows the implication $(b)\implies (c)$ does not hold in general (see Example \ref{exam: bnotc}).  Thus the notions of existence of a free subalgebra and of existence of a homogeneous free subalgebra and of existence of a monomial free subalgebra are pairwise inequivalent.

The outline of this paper is as follows. In \S 2, we give an easily verified criterion for constructing graded nilpotent monomial algebras that are not nil.  In \S 3 we prove Theorem \ref{Graded_Nil_Free_Subalgebra}. Our construction involves taking a graded nilpotent algebra $A$ that is not nil that is constructed in \S 2.  Then the free product of $A$ with itself contains a free algebra on two generators since $A$ is not nil.  We show that by taking an appropriate homomorphic image of this free product, we obtain an algebra that is graded nilpotent but that still contains a free subalgebra on two generators. Finally, in \S 4 we discuss a not-so-well known example of Rowen's \cite{Rowen} and make the remark that Rowen's construction yields graded nil, finitely generated algebras that are not nil. Moreover, we interpret Rowen's algebra as the monomial algebra attached to the Thue-Morse sequence.
\section{Constructing graded nilpotent algebras}

In this section we give a general method of constructing graded nilpotent algebras. This is done via the use of \emph{iterative algebras}, which are a class of monomial algebras that were introduced by the first-named author and Madill \cite{Iterative}.  We quickly recall the construction here.

Let $\Sigma$ be a finite alphabet and let $\Sigma^*$ denote the free monoid on $\Sigma$, with $\varepsilon$ denoting the empty word.  Then a monoid endomorphism $\phi :\Sigma^*\to \Sigma^*$ is completely determined by the images of the elements of $\Sigma$.  We say that $a\in\Sigma$ is \emph{mortal} if there is some $j$ such that $\phi^j(a)=\varepsilon$, and we let $\Gamma$ denote the submonoid of $\Sigma^*$ generated by the mortal letters.  

The morphism $\phi$ is \emph{prolongable} on $b\in \Sigma$ if $\phi(b)=bx$ with $x\in \Sigma^*\setminus \Gamma$.  In the case that we have such a $\phi$, we can construct a right-infinite word
$$w=bx\phi(x)\phi^2(x)\ldots $$ over the alphabet $\Sigma$, which is the unique right-infinite fixed point of $\phi$ starting with the letter $b$.  Given a field $k$, we let $k\langle \Sigma\rangle$ denote the free $k$-algebra on the alphabet $\Sigma$. An iterative algebra is then the algebra $k\langle \Sigma\rangle/I,$ where $I$ is the ideal generated by all words in $\Sigma^*$ that do not occur as a subword of $w$.  We denote this algebra by $A_w$, where the base field $k$ is understood.  If $\Sigma=\{x_1,\ldots ,x_d\}$, then we can give the algebra $A_w$ an $\mathbb{N}$-grading by declaring that $x_i$ has degree $p_i$ for some $p_i>0$.  We call the row vector $[p_1,\ldots ,p_d]$ the \emph{weight} vector associated to our grading.  

A distinguished subclass of monoid endomorphisms are the so-called \emph{primitive} morphisms. These are morphisms $\phi$ with the property that for each pair of letters $a,b\in \Sigma$ there is some $n$ such that $b$ occurs in $\phi^n(a)$. 

Given our prolongable endomorphism $\phi$ of the alphabet $\Sigma=\{x_1,\ldots ,x_d\}$, we can associate a $d\times d$ matrix, $M(\phi)$, which is called the \emph{incidence matrix}.  The $(i,j)$-entry of $M(\phi)$ is the number of occurrences of $x_i$ in $\phi(x_j)$.  Given a word $u\in \Sigma^*$, we can then associate an $m\times 1$ integer vector $\theta(u)$ whose $j$-th coordinate is the number of occurrences of $x_j$ in $u$. Then \cite[Proposition 8.2.2]{AS} shows that we have the relationship
\begin{equation}
\theta(\phi^n(u)) = M(\phi)^n \theta(u).\label{eq: theta}
\end{equation}
Ultimately, we will think of giving the letters $x_1,\ldots ,x_d$ weights and so if $x_i$ has degree $p_i$ for $i=1,\ldots ,d$ then $[p_1,\ldots ,p_d]\cdot \theta(u)$ is simply the weight of $u$ and 
$[p_1,\ldots ,p_d]M(\phi)^n \theta(u)$ is consequently the weight of $\phi^n(u)$.  

Our main result of this section is the following proposition, which gives a criterion for guaranteeing that an iterative algebra is graded nilpotent.  

\begin{prop} Let $\Sigma=\{x_1,\ldots ,x_d\}$ be a finite alphabet, let $\phi:\Sigma^*\to \Sigma^*$ be a primitive morphism that is prolongable on $x_1$, and let $w$ be the unique right-infinite word whose first letter is $x_1$ that is fixed by $\phi$.  Assume, in addition, that the following hold:
\begin{enumerate} 
\item we can write $w=ux_1w'$ where $w'$ is a right-infinite word and $u$ is a finite non-empty word;
\item there exist positive integers $p_1,\ldots ,p_d$ that are not all the same such that there is no $N>1$ such that the integers in the sequence $$[p_1,\ldots ,p_d]M(\phi)^j \theta(u)$$ for $j=0,1,\ldots $ are all divisible by $N$;
\item $\det(M(\phi))\in \{-1,1\}$.
\end{enumerate}
Then the positive part of the algebra $A_w$, in which $x_i$ is given degree $p_i$ for $i=1,\ldots ,d$, is graded nilpotent 
\label{prop: grnilpotent}
\end{prop}
We remark that in the statement of Proposition \ref{prop: grnilpotent}, item (1) always holds since we are assuming that $\phi$ is a primitive morphism; that is, there is always a decomposition of this form and one can use whichever decomposition of this form that one wants.  We only use this first condition to define $u$, which is used in Condition (2).  Condition (2) is fairly relaxed and tends to hold in practice. On the other hand, to satisfy condition (3), one generally has to exercise some care when defining the morphism to get the determinant condition.  We remark, however, that one requires something like condition (3) in general to produce graded nilpotent algebras. For example, if one takes the primitive morphism $\phi: \{x,y\}^*\to \{x,y\}^*$ which sends $x$ to $xy$ and $y$ to $yx$ then one obtains a right-infinite word $w=xyyxyxxy\cdots$. If one gives $x$ weight $p_1=1$ and $y$ weight $p_2=2$ and takes $u=xyy$, then conditions (1) and (2) hold. The determinant of $M(\phi)=0$, so condition (3) does not hold.  Notice that by definition of $\phi$, $w$ is a right-infinite word over the alphabet $\{xy,yx\}$ and since both of these words have degree three, the positive part of $A_w$ is not graded nilpotent.

We remark that Proposition \ref{prop: grnilpotent}, while technical, is easy to verify in practice.  We give a few quick examples to illustrate how one would use it in practice.
\begin{exam} Let $\Sigma=\{x,y\}$ and let $\phi:\Sigma^*\to \Sigma^*$ denote the morphism given by

$\phi(x)=xy$ and $\phi(y)=y^2x$.  Then if $w=x\phi(y)\phi^2(y)\phi^3(y)\cdots = xyyxyyxyyxxy\cdots$, then the positive part of $A_w$ is graded nilpotent when $x$ is given degree one and $y$ is given degree two.

\label{exam: xy}

\end{exam}

To see why this is the case observe that $\phi$ is primitive since both $x$ and $y$ occur in $\phi(x)$ and $\phi(y)$.  We can write $w=uxw'$ where $u=xyy$ and $w'$ is right-infinite.  We have $v:=\theta(u)=[1, 2]^T$.  Since $x$ and $y$ are being given degrees one and two respectively we have $p_1=1$ and $p_2=2$.  Finally,

$$ M(\phi) = \left( \begin{matrix}  1 & 1 \\ 1 & 2\end{matrix}\right),$$ which has determinant $1$.

We see $[p_1,p_2]M(\phi)^0 v = 5$ and $[p_1,p_2]M(\phi) v =13$ and since $5$ and $13$ have greatest common divisor one, Proposition \ref{prop: grnilpotent} gives that the positive part of $A_w$ is graded nilpotent.

\begin{exam} Let $\Sigma=\{x,y,z\}$ and let $\phi:\Sigma^*\to \Sigma^*$ denote the morphism given by

$\phi(x)=xyz$ and $\phi(y)=zx$, and $\phi(z)=yz$.  Then if $w=x\phi(yz)\phi^2(yz)\phi^3(yz)\cdots = xzxyzyzxyzzxyz\cdots$, then the positive part of $A_w$ is graded nilpotent when $x$ is given degree $1$, $y$ is given degree $2$, and $z$ is given degree $3$.

\label{exam: xyz}

\end{exam}

First, $\phi$ is primitive since $\phi(x)$ contains all letters, $\phi^2(y) = yzxyz$ contains all letters, and $\phi^3(z)$ contains all letters. We can write $w=uxw'$ where $u=xz$ and $w'$ is right-infinite.  Then $v:=\theta(u)=[1, 0,1]^T$.  Our vector of weights is given by $[p_1,p_2,p_3]=[1,2,3]$.

We also have $$ M(\phi) = \left( \begin{matrix}  1 & 1 & 0 \\ 1 & 0 & 1 \\ 1 & 1 & 1\end{matrix}\right), $$ which has determinant $-1$; moreover, we see $[p_1,p_2,p_3]M(\phi)^0 v = 4$ and $[p_1,p_2]M(\phi) v =9$ and since $4$ and $9$ have greatest common divisor one, Proposition \ref{prop: grnilpotent} gives that the positive part of $A_w$ is graded nilpotent.

Before we give the proof of Proposition \ref{prop: grnilpotent}, we require a simple combinatorial lemma.
\begin{lem}
\label{lem: progression}
Let $D>1$ be a positive integer, let $0\le i_1<i_2<\cdots <i_q<D$ be a sequence of positive integers and let $\mathcal{S}$ be the union of the arithmetic progressions $\{a+Dn\colon n\ge 0\}$ with $a\in \{i_1,i_2,\ldots ,i_q\}$.  Suppose that $\mathcal{S}$ does not contain any infinite arithmetic progressions of difference strictly less than $D$.  Then the sequence $$(i_2-i_1,i_3-i_2,i_4-i_3,\ldots ,i_q-i_{q-1},i_1+D-i_q)$$ is not fixed by a non-trivial cyclic permutation of itself.
\end{lem}
\begin{proof}
Let $b_j:=i_{j+1}-i_{j}$ for $j=1,\ldots ,q$, where we take $i_{q+1}=i_1+D$, and let 
$$
{\bf b}:=(b_1,b_2,\dots, b_q)\in \N^q.
$$
We let $\sigma$ be the $q$-cycle $(1,2,3,\dots,q)\in S_q$ and for $\pi\in S_q$ we let  $\pi({\bf b})$ denote $(b_{\pi(1)}, b_{\pi(2)},\dots, b_{\pi(q)})$.

Suppose that a non-trivial cyclic permutation of ${\bf b}$ is equal to ${\bf b}$.  Then $\sigma^m({\bf b})={\bf b}$ for some $m\in \{1,\ldots ,q-1\}$. Let $\pi=\sigma^m$ so that we then have that $b_i=b_{\pi(i)}$. By definition, $\mathcal{S}$ is equal to the set
$$
\{i_1,i_2,\dots, i_q, i_1+D,i_2+D,\dots, i_q+D, i_1+2D,\dots\}
$$ 
Moreover, the differences between successive terms of this sequence are given by the sequence
$$
(b_1,b_2,\dots, b_q, b_1,b_2,\dots, b_q,b_1,\dots).
$$
By repeatedly applying the identity $\sigma^m({\bf b})={\bf b}$, we have that
$$
(b_1,b_2,\dots, b_q, b_1,b_2,\dots, b_q,b_1,\dots)=(b_1,b_2,\dots, b_{m},b_1,b_2,\dots, b_{m},b_1,\dots).
$$
Therefore $\mathcal{S}$ contains an infinite arithmetic progression of the form $i_1+(b_1+\dots +b_{m})\N$, and $b_1+\cdots +b_m=i_{m+1}-i_1<D$, contradicting our assumption on $\mathcal{S}$.  The result follows.
\end{proof}

\begin{proof}[Proof of Proposition \ref{prop: grnilpotent}]
Write $w=a_1a_2a_3\cdots$, where $a_i\in \Sigma$ is the $i$-th letter of $w$. 
We let $\mathcal{S}=\{s_0,s_1,s_2,\ldots\}$ denote the subset of $\mathbb{N}_0$ in which $s_0=0$ and for $i\ge 1$, $s_i=s_{i-1}+{\rm deg}(a_i)$.  That is,
$s_i = \sum_{j=0}^i {\rm deg}(a_i)$.  Now if the positive part of $A_w$ is not graded nilpotent then there must exist some fixed $D>0$ such that there are arbitrarily long subwords of $w$ of the form $u_1u_2\cdots u_p$ with each $u_i$ having degree $D$.  This then gives that there exist arbitrarily long arithmetic progressions of difference $D$ in $\mathcal{S}$.  In particular, $\mathcal{S}$ must contain arithmetic progressions of the form $a,a+D,a+2D,\ldots ,a+(r-1)D$ for every $r\ge 1$.  

Then for every $r\ge 2$ there exists a subword $u_1u_2\cdots u_{r}$ of $w$ such that each $u_i$ is a subword of $w$ of degree $D$. Since $\phi$ is a primitive morphism, there exists some fixed $C>0$ such that whenever $x_1\in \Sigma$ occurs in $w$, its next occurrence in $w$ is at most $C$ positions later (see Allouche and Shallit \cite[Theorem 10.9.5]{AS}, which shows that for every finite subword $z$ of $w$, one has a constant $C=C(z)$ such that whenever $z$ occurs in $w$ its next occurrence in $w$ is at most $C$ positions later).  In particular for every $n\ge 1$, we see that if we take $r$ sufficiently large there is a subword of $u_1u_2\cdots u_r$ of the form $\phi^n(x_1)$.  

It follows that 
there exist natural numbers $i$ and $j$ with $i>1$ and $i+j<r$ such that $\phi^n(x_1)=v u_i\cdots u_{i+j} v'$, where $v$ is a proper (possibly empty) suffix of $u_{i-1}$ and $v'$ is a proper (possibly empty) prefix of $u_{i+j+1}$. Since $\phi^{n}(x_1)$ is a prefix of $w$ and it begins $vu_iu_{i+1}\cdots$, we see that $\mathcal{S}$ must contain a progression of the form $t,t+D,t+2D,\ldots ,t+sD$ with $t<D$ and such that $i+(s+1)D$ is strictly greater than the length of $\phi^{n}(x_1)$.  By assumption, $\mathcal{S}$ contains arbitrarily long arithmetic progressions of length $D$ and so there are infinitely many natural numbers $n$ for which 
$\phi^n(x_1)$ contains a progression of the form $t,t+D,t+2D,\cdots t+sD$ for some $t<D$ and $t+(s+1)D$ strictly greater than the length of $\phi^n(x_1)$.  Thus there is some fixed $t<D$ for which there are infinitely many natural numbers $n$ with this property.  Hence $\mathcal{S}$ contains arbitrarily long arithmetic progressions of the form $t,t+D,t+2D,\ldots ,t+rD$ for some fixed $t<D$, and so $\mathcal{S}$ contains an infinite arithmetic progression $t,t+D,t+2D,\ldots $.  

We now suppose towards a contradiction that the positive part of $A_w$ is not graded nilpotent and pick $D\in \N$ minimal with respect to having the property that $\mathcal{S}$ has an infinite arithmetic progression of the form $a+D\N$, where $a< D$.  Then $D>1$ since by condition (2) some $x_i$ has degree $p_i>1$ and since our morphism is primitive all letters occur infinitely often.  Moreover, $\mathcal{S}$ cannot contain any infinite arithmetic progressions with difference $<D$, because if it has such a progression then by the argument we just gave we see that it has a progression of the form $a+e\N$, where $a< e<D$, which contradicts the minimality of our choice of $D$.

We define
$$
T:=\lbrace a: 0\leq a< D, \{a+Dn\colon n\ge 0\}\subseteq  \mathcal{S}\rbrace.
$$
We write $T=\lbrace i_1, i_2, \dots, i_{q}\rbrace,$ where $0\le i_1<i_2<\dots<i_{q}<D$. Notice that there exists a positive integer $N$ such that if $j\in \{0,\ldots ,D-1\}\setminus T$ then there is some $m$, depending upon $j$, such that $j+mD\le N$ and $j+mD\not\in \mathcal{S}$.  Then by Lemma \ref{lem: progression}, since $\mathcal{S}$ does not contain any arithmetic progressions of length $<D$ we see that no non-trivial cyclic permutation of $$(i_2-i_1,i_3-i_2,\ldots ,i_q-i_{q-1},i_1+D-i_q)$$ can be equal to itself.

Now by (1) we have $w=ux_1 w'$ with $w'$ right-infinite and $u$ a word beginning with $x_1$. Let $v=[c_1,\ldots ,c_d]^T$ denote the column vector whose $j$-th coordinate is the number of occurrences of $x_j$ in $u$ (that is, $v=\theta(u)$) and we let $p=[p_1,\ldots ,p_d]$ denote the row vector whose $j$-th coordinate is the degree of $x_j$.  Then since $w$ is a fixed point of $\phi$, for each $n$ we have $\phi^n(u)\phi^n(x_1)$ is a prefix of $w$.  Recall that the degree of $\phi^n(u)$ is given by $pA(\phi)^n v$ by Equation (\ref{eq: theta}).  We claim that $\phi^n(u)$ must have degree equal to a multiple of $D$ for every $n$.  To see this, let $\ell_n$ denote the weight of $\phi^n(u)$ and let $\ell'_n$ denote the weight of $\phi^n(x_1)$ for each $n\ge 0$.  


Then given a positive integer $n> N$, there exists a unique $s\in \{1,\ldots ,q\}$ and a unique $r\ge 0$ such that $i_s+Dr$ is the largest positive integer in the set $\{i_j+D \ell \colon 1\le j\le q, \ell\ge 0\}$ that is less than or equal to $\ell_n$.  Since $\phi^n(u)\phi^{n}(x_1)$ is a prefix of $w$, we see that the part of the set
$$\{i_1,i_2,\ldots ,i_q, i_1+D,i_2+D\ldots ,i_s+Dr,\ell_n+i_1, \ell_n+i_2,\ldots , \ell_n+i_q,\ell_n+i_1+D,\ldots \}$$
in 
$[0,\ell_{n}+\ell'_{n}]$ is entirely contained in $\mathcal{S}$.

For $j=1,\ldots ,q$, define $i_{s+j}$ to be $i_{s+j-q}+D$ if $s+j>q$.  Then by definition of $T$, we have that
$$\{i_{s+1}+Dr, i_{s+2}+Dr,\ldots , i_{s+q}+Dr,i_{s+1}+D(r+1),\ldots \}\cap [0,\ell_{n}+\ell_{n}']\subseteq \mathcal{S}$$ and so subtracting $\ell_n$, the weight of $\phi^n(u)$, and using the fact that the prefix $\phi^n(u)$ in $w$ is then followed by $\phi^{n}(x_1)$ in $w$, we see that 
$$\{i_{s+1}+Dr-\ell_n,i_{s+2}+Dr-\ell_n,\ldots\}\cap [0,\ell_{n}']\subseteq \mathcal{S}.$$
Since $n-1\ge N$ and $i_{s+j}+Dr-\ell_n\in \{0,\ldots ,D-1\}$ for $j=0,\ldots ,q$, we see from our choice of $N$ that we must have $i_{s+j}+Dr-\ell_n = i_j$ for $j=1,\ldots ,q$.  
Notice also that $(i_{s+j+1}+Dr-\ell_n)-(i_{s+j}+Dr-\ell_n)=i_{s+j+1}-i_{s+j}$ for $j=1,\ldots ,q$, where we take $i_{s+j}=i_{s+j-q}$ if $s+j>q$.  Since 
$i_{s+j}+Dr-\ell_n = i_j$, we see that $i_{j+1}-i_{j}=i_{s+j+1}-i_{s+j}$ for $j=1,\ldots ,q$.  Since no non-trivial cyclic permutation of $$(i_2-i_1,i_3-i_2,\ldots ,i_q-i_{q-1},i_1+D-i_q)$$ can be equal to itself, we have that $s=q$ and so taking $j=1$ in the equation $i_{s+j}+Dr-\ell_n = i_j$ gives
$i_1+Dr+D-\ell_n = i_1$.  In particular, $\ell_n\equiv 0~(\bmod~D)$ for all $n> N$.  

Since $\ell_n = p A^n v$ and since $D>1$, we see that if the positive part of $A_w$ fails to be graded nilpotent then for all $n>N$ we must have that $D$ divides $\ell_n$.   In particular, since $D>1$ there is some prime number $t$ such that $t|\ell_n$ for all sufficiently large $n$. By condition (2), not all $\ell_i$ are divisible by $t$, and so we see that there is some largest natural number $m$ such that $t$ does not divide $\ell_m$.  But now by Condition (3), the Cayley-Hamilton theorem gives a relation of the form $A^d + c_{d-1}A^{d-1}+\cdots + c_0I =0$, where $c_0=\pm \det(A)=\pm 1$.  So if we multiply this relation by $A^{m}$, we see that $A^m$ can be expressed as an integer linear combination of $A^{m+1},\ldots ,A^{m+d}$ and consequently $\ell_m$ can be expressed as an integer linear combination of $\ell_{m+1},\ldots ,\ell_{m+d}$.  But this is a contradiction, since $t|\ell_{m+j}$ for all $j>0$ and $t$ does not divide $\ell_m$.  The result follows.
\end{proof}

We remark that Proposition \ref{prop: grnilpotent} gives a fairly general criterion for guaranteeing when an iterative algebra is graded nilpotent with respect to a given grading. It does not, however, answer the general question of which iterative algebras have a grading that ensures one obtains a graded nilpotent ring. We thus ask the following question.
\begin{ques} Let $k$ be a field and suppose that $A_w=k\{x_1,\ldots ,x_d\}/I$ is an iterative algebra associated to the endomorphism $\phi: \{x_1,\ldots ,x_d\}^*\to \{x_1,\ldots ,x_d\}^*$. Can one give a concrete characterization of the morphisms $\phi$ that are prolongable on $x_1$ and have the property that there exist $p_1,\ldots ,p_d>1$ such that when $x_i$ is given degree $p_i$ the algebra $A_w$ is locally nilpotent?
\end{ques}
Since this question could potentially be of interest to people whose chief research area is combinatorics of words, we rephrase the question in this language.  Here, one wishes to characterize the morphisms of finite alphabets $\{x_1,\ldots ,x_d\}$ that yield pure morphic right-infinite words $w$ with the property that there are positive integer weights $p_1,p_2,\ldots ,p_d$ such that the sequence $${\rm wt}(w_1),{\rm wt}(w_2),{\rm wt}(w_3),\ldots$$ does not contain arbitrarily long arithmetic progressions of difference $D$ for any $D>1$, where $w_i$ is the first $i$ letters of $w$ and if $v$ is a finite word then ${\rm wt}(v)$ is the sum of the weights of the letters of $v$.

We would like to remark that in \cite{AgataGrNil} it is proved (Theorem 1.1) that for a graded algebra $R$ generated in degree $1$, being Jacobson radical is equivalent to having the following property: for all $n,d\in \mathbb{N}$, the space $M_n(R_d)$ of matrices of size $n$ with entries belong homogeneous elements of degree $d$ is nil. It is mentioned there that it is unknown if this equivalence is true for graded algebras not necessarily generated in degree $1$.

Our remark is that every graded nilpotent algebra satisfies the condition that $M_n(R_d)$ is nil, but such algebras need not be Jacobson radical (as was already shown in \cite{Iterative}).

\section{Graded Nilpotent Algebras with Free Subalgebras}

The main objective of this section is to prove Theorem \ref{Graded_Nil_Free_Subalgebra}.  To better describe our construction, we give some notation. Given a finite alphabet $\Sigma$ and a right-infinite word $v=a_1a_2a_3\cdots$, with $a_1,a_2,\ldots \in \Sigma$, we define $v[i,j]:=a_i\cdots a_j$.

In order to give our construction, we require a simple lemma.
\begin{lem} \label{sequence}
There exists an increasing sequence of nonnegative integers $0=n_0<n_1<n_2<\cdots$ with the property that for every finite sequence of positive integers $(a_1,\dots,a_s)$ there exists some natural number $m$ such that $a_1=n_{m+1}-n_m,\ldots,a_s=n_{m+s}-n_{m+s-1}$.
\end{lem}

\begin{proof} Given two finite sequences ${\bf a}=(a_1,\ldots ,a_s)$ and ${\bf b}=(b_1,\ldots ,b_t)$ of positive integers, we let 
${\bf a}\star {\bf b} = (a_1,\ldots ,a_s,b_1,\ldots ,b_t)$.  We note that $\star$ is associative and given a countable collection of finite sequences of positive integers, ${\bf a}_1, {\bf a}_2, \ldots$, we can form an infinite sequence by taking the limit of ${\bf a}_1\star {\bf a}_2 \star \cdots {\bf a}_n$ as $n\to \infty$, which we denote by ${\bf a}_1\star {\bf a}_2 \star {\bf a}_3 \star \cdots$.

Since there are only countably many finite sequences of positive integers, we may enumerate them, ${\bf a}_1, {\bf a}_2, \ldots$.  We then let ${\bf b}$ denote the infinite sequence formed by concatenation of them; that is $${\bf b}= {\bf a}_1\star {\bf a}_2 \star \cdots.$$ We now write ${\bf b}=(b_1,b_2,b_3,\ldots )$, where each $b_i$ is a positive integer, and we set $n_0=0$ and for $i>0$ we let $n_i=b_1+\cdots+b_i$. Then the sequence $n_0<n_1<n_2<\cdots$ has the desired property since each finite sequence ${\bf a}_i$ appears in the sequence of differences, by construction.
\end{proof}

We now give our main construction of this paper. By Example \ref{exam: xy} we have a right-infinite word \begin{equation}
\label{eq: ww}
w=xyyyx\cdots
\end{equation} over the alphabet $x,y$ and a monomial algebra $A_w$ with the property that, when $x$ is given degree one and $y$ is given degree two, it is graded nilpotent.  We let $w'=x'y'y'y'x'\cdots $ be the right-infinite word over the alphabet $x',y'$ obtained by making the substitutions $x\mapsto x'$ and $y\mapsto y'$ in $w$.  



We now let $0=n_0<n_1<\cdots $ be an increasing sequence of nonnegative integers satisfying the property from the statement of Lemma \ref{sequence} and we make a right-infinite word $\tilde{w}$ over the alphabet $\{x,y,x',y'\}$ defined by
\begin{equation} \label{eq: tilde}
\tilde{w}=w[1,n_1]w'[n_1+1,n_2]w[n_2+1,n_3]w'[n_3+1,n_4]w[n_4+1,n_5]\cdots.\end{equation}

We are now ready to give our construction.  We let
\begin{equation} A:=A_{\tilde{w}}=k\langle x,y,x',y'\rangle/I,\label{eq: monom}
\end{equation}
where $I$ is the ideal generated by all monomials which do not appear as a subword of $\tilde{w}$. 
\begin{thm} Let $A$ be as in Equation (\ref{eq: monom}) and endow $A$ with a grading by declaring that $\deg(x)=\deg(x')=1$ and $\deg(y)=\deg(y')=2$.  Then the positive part of $A$ is graded nilpotent and the elements $x+y$ and $x'+y'$ generate a free subalgebra of $A$.
\label{thm: construction}
\end{thm}
\begin{proof}
We first claim that the positive part of $A$ is graded nilpotent.  To see this, let $d$ be a positive integer.  Since the homogeneous elements of degree $d$ in $A$ are spanned by a finite set of monomials, it is sufficient to show that if $u_1,\ldots ,u_m$ are monomials in $A$ of degree $d$ then there is some $N=N(m)$ such that $u_{i_1}u_{i_2}\cdots u_{i_N}=0$ for every $i_1,\ldots ,i_N\in \{1,\ldots ,m\}$.  If there is no such $N$ then there are arbitrarily long subwords of $\tilde{w}$ of the form $u_{j_1}u_{j_2}\cdots u_{j_p}$, where $\tilde{w}$ is as in Equation (\ref{eq: tilde}).  Now we have a semigroup morphism $\phi: \{x,y,x',y'\} \to \{x,y\}$ given by sending $x$ and $x'$ to $x$ and $y$ and $y'$ to $y$.  We can extend $\phi$ to right-infinite words and by construction we have $\phi(\tilde{w})$ is equal to $w$, where $w$ is as in Equation (\ref{eq: ww}).  Then $\phi(u_1),\ldots ,\phi(u_m)$ are words in $x$ and $y$ of degree $d$ and we have that $w$ contains arbitrarily long subwords of the form $\phi(u_{j_1})\phi(u_{j_2})\cdots \phi(u_{j_p})$, which contradicts the results from Example \ref{exam: xy}.  

We now show that $u:=x+y$ and $v:=x'+y'$ generate a free subalgebra of $A$.  To see this, suppose that there exists some nonzero relation $f(X,Y)\in k\langle X,Y\rangle$ for which $f(u,v)=0$.  By left-multiplying our relation by $X$ and right-multiplying by $Y$, we may assume that 
we can write $$f(X,Y) = \sum c_{i_1,j_1,\ldots ,i_n,j_n} X^{i_1} Y^{j_1} X^{i_2} \cdots Y^{j_n},$$ where the sum runs over all sequences of positive integers $(i_1,j_1,\ldots ,i_n,j_n)$ of even length and such that $c_{i_1,j_1,\ldots ,i_n,j_n}=0$ for all but finitely many sequences.


Pick a sequence $(i_1,j_1,\ldots ,i_n,j_n)$ with the property that
$c_{i_1,j_1,\ldots ,i_n,j_n}\neq 0$. 
Since $u$ involves only the letters $x$ and $y$ and $v$ involves only the letters $x'$ and $y'$, we see that $u^{i_1} v^{j_1} u^{i_2} \cdots v^{j_n}$ is spanned by words over $x,x',y,y'$ such that the first $i_1$ letters are from $\{x,y\}$, the next $j_1$ letters are from $\{x',y'\}$, and so on.  On the other hand, if $(i_1',j_1',\ldots ,i_m',j_m')$ is a different sequence then 
$u^{i_1'} v^{j_1'} \cdots u^{i_m'} v^{j_m'}$ does not involve any words of that form.  Since $A$ is a monomial algebra, we then see that if $u$ and $v$ fail to generate a free algebra then we must have a non-trivial relation of the form
$$u^{i_1} v^{j_1} u^{i_2} \cdots u^{i_r} v^{j_r}=0.$$

By construction we have
$$u^{i_1} v^{j_1} u^{i_2} \cdots v^{j_r}=\sum W_1 V_1 W_2 V_2 \cdots W_r V_r,$$ where the sum on the right-hand side ranges over all $(W_1,V_1,\ldots ,W_r,V_r)$ where, for each $p$, $W_p$ is a word on the alphabet $x,y$ of length $i_p$ and $V_p$ is a word on the alphabet $x',y'$ of length $j_p$.
Since $A$ is a monomial algebra we then see that every word $$W_1 V_1 W_2 V_2 \cdots W_r V_r$$ of this form must have zero image in $A$ and so by definition of $A$, we see that no word of this form can be a subword of $\tilde{w}$.  

But by construction we have some even number $m$ for which 
\begin{eqnarray*}
i_1=n_{m+1}-n_m, j_1=n_{m+2}-n_{m+1},i_2=n_{m+3}-n_{m+2},\ldots \\
 i_r=n_{m+2r-1}-n_{m+2r-2},j_r=n_{m+2r}-n_{m+2r-1},\end{eqnarray*} so the word $$w[n_m+1, n_m+i_1] w'[n_{m+1}+1, n_{m+1}+j_1]\cdots $$ appears as a subword of $\tilde{w}$ and has the form mentioned above. Therefore we see that $u$ and $v$ have no non-trivial relations and so the elements $u$ and $v$ generate a free subalgebra.
\end{proof}

We end this section by giving a quick example of a monomial algebra with a free subalgebra generated by two homogeneous elements of degree one (in particular it has exponential growth), but in which every monomial is nilpotent; in particular, it cannot contain a free subalgebra generated by monomials and so this shows that the implication $(b)\implies (c)$ need not hold in general, where ($a$), ($b$), ($c$) are the statements defined in the Introduction.

\begin{exam}
\label{exam: bnotc}
Let $k$ be a field and let $A=k\langle x,y,z,w\rangle/I$ where $I$ is the ideal generated by the cubes of all monomials of length $\ge 1$.  Then the cube of every non-trivial monomial has zero image in $A$ but the subalgebra of $A$ generated by $x+y$ and $z+w$ is free.
\end{exam}

\begin{proof}
Suppose, towards a contradiction, that we have a non-trivial relation between $u:=x+y$ and $v=z+w$.  
Then using the same argument that was employed in the proof of Theorem \ref{thm: construction}, we see that if $u$ and $v$ fail to generate a free algebra then we must have a non-trivial relation of the form
$$u^{i_1} v^{j_1} u^{i_2} \cdots u^{i_r} v^{j_r}=0$$ for some $r\ge 1$ and some positive integers $i_1,j_1,\ldots ,i_r,j_r$.

Notice that by construction we have
$$u^{i_1} v^{j_1} u^{i_2} \cdots v^{j_r}=\sum W_1 V_1 W_2 V_2 \cdots W_r V_r,$$ where the sum on the right-hand side ranges over all $(W_1,V_1,\ldots ,W_r,V_r)$ where, for each $p$, $W_p$ is a word on the alphabet $x,y$ of length $i_p$ and $V_p$ is a word on the alphabet $z,w$ of length $j_p$.

Since $A$ is a monomial algebra we then see that every word $$W_1 V_1 W_2 V_2 \cdots W_r V_r$$ of this form must have zero image in $A$ and so by definition of $A$, we see that every word of this form must contain the cube of a monomial as a subword.  

Now let $w$ be the right-infinite word that begins with $x$ that is the unique fixed point of the morphism $\phi:\{x,y\}^*\to \{x,y\}^*$ given by $x\mapsto xy$ and $y\mapsto yx$ ($w$ is often called the \emph{Thue-Morse word}).  Let $v$ be the right-infinite word over $\{z,w\}$ obtained by making the letter-by-letter substitutions $x\mapsto z$ and $y\mapsto w$.  Then neither $w$ nor $v$ contain the cubes of a non-trivial monomial as a subword (cf. Allouche and Shallit \cite[Theorem 1.8.1]{AS}).  In particular, the monomial $w[1,i_1]v[1,j_1]w[1,i_2]v[1,j_2]\cdots w[1,i_r]v[1,j_r]$ has nonzero image in $A$ and so by the above remarks $u^{i_1} v^{j_1} u^{i_2} \cdots v^{j_r}$ is nonzero and is of the form given above and so we see that $u$ and $v$ generate a free subalgebra of $A$.
\end{proof}

\section{Rowen's Example}

We make a brief remark about an example of Rowen \cite{Rowen}. Our reasons for doing so are threefold; first, Rowen's example is unknown by many in the ring theory community due to being published in a conference proceedings; second there is some relevance with the theme of graded nil algebras that are not nil, which is the underlying theme of this paper; finally, Rowen's example itself has an interesting connection to the iterative algebra construction, which we shall see.

\subsection{The construction} \label{Rowen's construction}

Recall the \textit{Thue-Morse sequence}, which is binary sequence

$$(m_1,m_2,m_3,\dots)=(1,0,0,1,0,1,1,0,0,1,1,0,1,0,0,1,\dots),$$ where $m_i$ is $1$ if the number of ones in the binary expansion of $i-1$ is even and is $0$ otherwise.  We leave it to the reader to prove that this sequence, when regarded as a right-infinite word over the alphabet $\{0,1\}$, is the same as the right-infinite word over $\{0,1\}$ beginning with $1$ that is a fixed point of the morphism $1\mapsto 10$ and $0\mapsto 01$. This shows there is some overlap with the construction of iterative algebras (see, in particular, Example \ref{exam: bnotc}, where the Thue-Morse word is used again).  Indeed the key property used in Rowen's construction is that the Thue-Morse word is cube free, which is used in Example \ref{exam: bnotc}.  

Let $k$ be a field and let $V=ke_1+ke_2+\cdots$ be a countably infinite-dimensional $k$-vector space.  We let $A$ denote the endomorphism ring of $V$, which we think of as being column-finite matrices with entries in $k$.  We let $e_{i,j}$ denote the element of $A$ which sends $e_{\ell}$ to $\delta_{j,\ell}e_i$ and we let $$a=\sum_{i=1}^{\infty} m_ie_{i,i+1}, \qquad b=\sum_{i=1}^{\infty} (1-m_i)e_{i,i+1}.$$

Rowen \cite[Example 1]{Rowen} proved that the subalgebra $S$ of $A$ generated by $a$ and $b$ has the property that $f(a,b)a$ and $f(a,b)b$ are nilpotent whenever $f(a,b)$ is a homogeneous polynomial in $a$ and $b$. In fact, this property follows from the following property of the Thue-Morse sequence: for every $u\geq 1$ there exists some $n_u\geq 1$ such that for every $i\geq 1$, we have that the set $\{m_{i},m_{i+u},\dots, m_{i+n_u u}\}$ contains both $0$ and $1$.

Now $V$ can be given a $\mathbb{Z}$-grading by declaring that $e_i$ has degree $-i$ for every $i\ge 1$.  Then we see that by definition, $a$ and $b$ send homogeneous elements of $V$ of degree $-i$ to elements of degree $-i+1$ and so this endows $S$ with an $\mathbb{N}$-grading, in which $a$ and $b$ have degree one. Moreover, Rowen's result shows that $Sa$ and $Sb$ are graded nil left ideals (and in fact, graded nilpotent).

Rowen's construction works over any base field.  In particular, if we now assume that $k$ is uncountable then the validity of the K\"othe problem for algebras over uncountable fields gives that if $Sa$ and $Sb$ are in fact nil left ideals (rather than simply graded nil left ideals) then $S_{>0}=Sa+Sb$ is nil, too. But $a+b=\sum_{i=1}^{\infty} e_{i,i+1}$, which is not a nilpotent endomorphism, so $Sa+Sb$ cannot be nil and so we see that at least one of $Sa$ or $Sb$ is graded nil but not nil.  Hence we have a graded nil algebra that is not nil. Although the rings $Sa$ and $Sb$ need not be finitely generated, it is immediate that if $c\in \{a,b\}$ is such that $Sc$ is graded nil and non-nil then there exists a finitely generated algebra $S_0\subseteq Sc$ that is graded nil algebra and that is not nil.

One may wonder what happens over countable base fields, where the definitions of `nil' and `Jacobson radical' do not coincide (for countably generated algebras over uncountable base fields they do). In fact, in the following subsection we show that $S$ is primitive.





We would now like to provide a historical remark. After observing that Rowen's example yields a graded nil non-Jacobson radical algebra, we were informed by Louis Rowen that during the conference in 1989 where the paper \cite{Rowen} was presented, Kaplansky pointed out to Rowen that his example in fact provides a graded nil but not nil algebra. We therefore wish to attribute this observation to Kaplansky.

In the next subsection we show that Rowen's example is in fact an iterative algebra, as defined in Section 2, namely a monomial algebra defined by the Thue-Morse sequence via forbidden words. As an application, we show that Rowen's example is a primitive algebra of quadratic growth.

Note that Rowen's example is not isomorphic to the algebra constructed in \cite{AgataGrNil}. Indeed, the example from \cite{AgataGrNil} is not monomial and has exponential growth, whereas Rowen's example is monomial of quadratic growth as we see in the next subsection.

\subsection{Monomial presentation and further properties} \label{TM}

Let $w$ be the right-infinite word over the alphabet $\{x,y\}$ that is the unique word whose first letter is $y$ and that is a fixed point of the morphism $x\mapsto xy$, $y\mapsto yx$.  Then $w$ is the Thue-Morse word and is identical to the infinite word obtained from the Thue-Morse sequence, in which we replace `0' by `$x$' and `1' by `$y$'. That is, $$w = yxxyxyyxxyyxyxxy\cdots.$$

Given an infinite word $v$ on some finite alphabet $\Sigma$, one can define the monomial algebra $A_v$ associated with it, as the monomial algebra obtained by taking the free algebra generated by $\Sigma$ and quotienting by the ideal generated by all monomials that do not appear as a subword of $v$. We call the algebra obtained in this manner from the word $w$ above the \emph{Thue-Morse iterative algebra}, and we shall denote this algebra by $A$.

Recall that an infinite word $w$ is said to be \textit{uniformly recurrent} if for every finite subword $u$ of $w$ there exists some constant $C=C_u>0$ such that whenever $u$ occurs in $w$, its next occurrence is at most $C$ positions later.






We are now ready to prove that Rowen's example $S$ is isomorphic to the Thue-Morse iterated algebra. For $\varepsilon$ in $\{0,1\}$, we use the notation $\varepsilon^m$ to denote the sequence $\varepsilon\cdots \varepsilon$ where there are $m$ occurrences of $\varepsilon$.

\begin{prop} \label{ScongA}
Let $k$ be a field and let $A$ be the Thue-Morse iterative algebra defined above and let $S$ be the algebra from Subsection \ref{Rowen's construction} (both over the field $k$). Then $A\cong S$ and $S$ is a prime just infinite algebra; that is, if $J$ is a nonzero ideal of $S$ then $S/J$ is finite-dimensional.
\end{prop}

\begin{proof}
There is a surjective homomorphism $\varphi:k\{x,y\} \rightarrow S$ given by $x\mapsto b$ and $y\mapsto a$. Consider a word $v:=y^{i_1}x^{j_1}\cdots x^{j_{n-1}}y^{i_n}x^{j_n}$ over $\{x,y\}$, where $i_1,j_n\ge 0$ and $j_1,i_2,\dots,j_{n-1},i_n>0$. Suppose that this word $v$ does not appear as a subword of the Thue-Morse word $w$ above. We claim that $\varphi(v)=0$.

Let $\alpha = \sum_{r=1}^{n}i_r+\sum_{r=1}^{n-1}j_r$ and write $a^{i_1}b^{j_1}\cdots b^{j_{n-1}}a^{i_n}b^{j_n}$ as an infinite matrix. 
$$\sum_{t=1}^{\infty} c_t e_{t,t+\alpha+j_n},$$
where  \begin{eqnarray*}
c_t&=& m_t\cdots m_{t+i_1-1}\cdot(1-m_{t+i_1})\cdots(1-m_{t+i_1+j_1-1})\cdots\\
&~& \cdots (1-m_{t+\alpha})\cdots(1-m_{t+\alpha+j_n-1})\end{eqnarray*}
 for $t\ge 1$.  
We note that this is equal to the zero matrix, since by assumption $v$ does not occur as a subword of the Thue Morse word and so $1^{i_1}0^{j_1}\cdots 0^{j_{n-1}}1^{i_n}0^{j_n}$ cannot occur as a subsequence of the Thue-Morse sequence and thus each $c_t=0$.  Hence $\varphi(v)=0$ whenever $v$ does not occur as a subword of the Thue-Morse word and thus $\varphi$ in fact induces a surjective homomorphism from $A$ to $S$.

Now since $w$ is uniformly recurrent and both $x$ and $y$ occur at least twice in $w$, we have that $A$ is a prime and just infinite \cite[Theorem 4.1]{Iterative}; that is, $A/J$ is finite-dimensional whenever $J$ is a nonzero ideal. Since $S$ is infinite-dimensional, we see that $A$ must be isomorphic to $S$. In particular, $S$ is also just infinite.  
\end{proof}

As a consequence, we have the following result.  We recall that a finitely generated algebra $B$ over a field $k$ has \emph{quadratic growth} if there is a finite-dimensional subspace $V$ of $B$ that generates $B$ as an algebra and that contains $1$ and there are positive constants $C_0,C_1$ such that $$C_0 n^2\le {\rm dim}(V^n) \le C_1n^2$$ for all sufficiently large $n$.

\begin{cor}
Let $k$ be a field and let $S$ be the $k$-algebra from Subsection \ref{Rowen's construction}. Then $S$ is a primitive algebra of quadratic growth.
\end{cor}

\begin{proof}
Using Proposition \ref{ScongA} it is equivalent to prove the assertion for $A$. To see that $A$ has quadratic growth, observe that the Thue-Morse word $w$ has linear subword complexity (e.g. see \cite{Complexity}) and so this shows that $A$ has quadratic growth (see \cite{Iterative} for more details on the connection between subword complexity and growth).

Since $A$ is a just infinite algebra, $A$ is prime \cite{FP}. Then by a result of Okni\'nski \cite{Trichotomy}, we have that $A$ is either primitive, satisfies a polynomial identity, or has nonzero locally nilpotent Jacobson radical.  Since $w$ is not eventually periodic, $A$ does not satisfy a polynomial identity \cite[Theorem 4.1]{Iterative}. If $J$ is the Jacobson radical of $A$ and suppose that $J$ is nonzero then $A/J$ is finite-dimensional since $A$ is just infinite.  Now $J$ is a homogeneous ideal \cite{Bergman} of finite codimension and so the image of $x+y$ in nilpotent in $A/J$. Thus some power of $x+y$ is in $J$ and since homogeneous elements of the Jacobson radical are nilpotent we then have that the image of $x+y$ in $A$ is nilpotent.  But this is impossible, since $A$ is an infinite-dimensional monomial algebra generated by $x$ and $y$.   We conclude that $A$ is primitive.
\end{proof}

We conclude with the following natural question.

\begin{ques}
Is there a finiteley generated graded nil algebra, generated in degree one which contains a free noncommutative subalgebra?
\end{ques}

It might be that a careful analysis of the example from \cite{AgataGrNil} would result in a modification which contains a free noncommutative subagebra.

\section*{Acknowledgments} 
We thank Blake Madill for many helpful comments and suggestions. We thank Louis Rowen both for the inspiration that his paper \cite{Rowen} gave us and for useful comments. We thank the referee for valuable remarks.

\end{document}